\documentclass[reqno,10pt]{amsart}
\usepackage{amscd,amssymb,amsmath,color,mathtools}
\usepackage{paralist}
\usepackage{latexsym}
\usepackage[all]{xy}
\usepackage[colorlinks=true]{hyperref}

\def\bfZ{{\bf Z}}
\def\bfR{{\bf R}}
\def\bfC{{\bf C}}

\def\calB{{\mathcal B}}
\def\bfm{{\mathbf{m}}}

\def\calC{{\mathcal C}}
\def\calL{{\mathcal L}}
\def\calO{{\mathcal O}}

\def\calF{{\mathcal F}}
\def\calP{{\mathcal P}}
\def\calS{{\mathcal S}}

\def\calZ{{\mathcal Z}}
\def\calX{{\mathcal X}}

\def\gp{{\rm gp}}
\def\ol{\overline}

\def\oM{{\ol{M}}}

\def\oR{{\ol{R}}}

\def\oQ{{\ol{Q}}}

\def\ophi{{\ol{\phi}}}

\def\oP{{\ol{P}}}

\def\rmlog{{\rm log}}
\def\..{{,\dots,}}
\def\:{{\colon}}

\def\int{{\rm int}}

\def\Spec{{\rm Spec}}

\def\toisom{{\xrightarrow\sim}}

\def\into{\hookrightarrow}
\def\toto{\rightrightarrows}

\newtheorem{theor}{Theorem}[section]
\newtheorem{prop}[theor]{Proposition}
\newtheorem{defn}[theor]{Definition}
\newtheorem{lem}[theor]{Lemma}

\newtheorem{conj}[theor]{Conjecture}

\def\Cay{{\mathcal{C}}}

\DeclareMathOperator{\conv}{conv}

\DeclareMathOperator{\Span}{Span}
\DeclareMathOperator{\AffSpan}{AffSpan}
\DeclareMathOperator{\Ob}{Ob}
\DeclareMathOperator{\supp}{supp}
\DeclareMathOperator{\ind}{index}
\DeclareMathOperator{\Cone}{Cone}

\theoremstyle{definition}
\newtheorem{defin}[theor]{Definition}
\newtheorem{rem}[theor]{Remark}
\newtheorem{exam}[theor]{Example}

\begin{document}

\author{Karim Adiprasito}
\author{Gaku Liu}
\author{Michael Temkin}

\date{\today}

\title{Semistable reduction in characteristic 0}

\address{Einstein Institute of Mathematics, The Hebrew University of Jerusalem, Giv'at Ram, Jerusalem, 91904, Israel}
\email{adiprasito@math.huji.ac.il}

\address{Max Planck Institute for Mathematics in the Sciences, 04103 Leipzig, Germany}
\email{gakuliu@gmail.com}

\address{Einstein Institute of Mathematics, The Hebrew University of Jerusalem, Giv'at Ram, Jerusalem, 91904, Israel}
\email{michael.temkin@mail.huji.ac.il}

\thanks{K.A. was supported by ERC StG 716424 - CASe and ISF Grant 1050/16. G.L. was supported by NSF grant 1440140.
M.T. was supported by ERC Consolidator Grant 770922 - BirNonArchGeom and ISF Grant 1159/15.
The authors thank Dan Abramovich for interest in this work, useful discussions and valuable comments.
}

\begin{abstract}
In 2000 Abramovich and Karu proved that any dominant morphism $f\:X\to B$ of varieties of characteristic zero can be made weakly semistable by replacing $B$ by a smooth alteration $B'$ and replacing the proper transform of $X$ by a modification $X'$. In the language of log geometry this means that $f'\:X'\to B'$ is log smooth and saturated for appropriate log structures. Moreover, Abramovich and Karu formulated a stronger conjecture that $f'\:X'\to B'$ can be even made semistable, which amounts to making $X'$ smooth as well, and explained why this is the best resolution of $f$ one might hope for. In this paper, we solve the semistable reduction conjecture in the larger generality of finite type morphisms of quasi-excellent schemes of characteristic zero.
\end{abstract}

\maketitle

\section{Introduction}

\subsection{The semistable reduction conjecture}
Resolution of singularities of an integral scheme $X$ is a classical and difficult problem of finding a modification $X'\to X$ with a regular source. The relative analogue seeks to improve a dominant morphism $f\:X\to B$ of integral schemes via a base change $B'\to B$ and a modification $X'$ of the proper transform $(X\times_BB')^{\rm pr}$ such that $f'\:X'\to B'$ is ``as smooth as possible''. Probably the main credit for coining the latter problem should be given to Mumford and the groundbreaking papers \cite{DM} and \cite{KKMS}, where $B$ was a trait, $B'\to B$ was a finite cover, and one achieved that $f'$ is {\em semistable}, that is, \'etale locally of the form $\Spec(k[t_0\ldots t_n]/(t_0\ldots t_m-\pi))$. In \cite{dejong-curves}, de Jong proved that semistable reduction of relative curves is possible for a wide class of bases $B$, including varieties, once one allows alterations of the base. (Note also in a somewhat different direction, one can use works of Alexeev, Koll\'ar and Shepherd-Barron (see \cite{alex1}, \cite{alex2}) to obtain a minimal model reduction theorem in the case of relative surfaces.)

To the best of our knowledge, the case of arbitrary $\dim(B)$ and $\dim(X)-\dim(B)$ was first studied by Abramovich and Karu in \cite{AK}. They showed that in the characteristic zero case at the very least the notion of semistability should be relaxed as follows: $X$ and $B$ are regular and possess local parameters $t_1\.. t_n$ and $\pi_1\..\pi_l$ such that $f^\#(\pi_i)=t_{n_i+1}\ldots t_{n_{i+1}}$ for $0=n_1<n_2<\ldots<n_{l+1}\le n$, see also \S\ref{semistabsec}. Then they proposed in \cite[Conjecture~0.2]{AK} a best possible conjecture for projective varieties over $\bfC$, which we formulate in the maximal expected generality of qe schemes:

\begin{conj}\label{conj:AK}
Let $f\:X\rightarrow B$ denote a dominant morphism of finite type of integral qe schemes. Then there is a projective alteration $B' \rightarrow B$, and a projective modification $X'\rightarrow (X\times_B B')^{\rm pr}$ such that $f'\:X' \rightarrow B'$ is semistable.
\end{conj}

\subsection{Previous works}
Abramovich and Karu proved a weak version of this conjecture for varieties of characteristic zero. First, they proved toroidal (or log smooth) reduction in \cite[Theorem~2.1]{AK}: $f'$ can be made toroidal. As in \cite{KKMS}, further improvements of $f'$ can be done by toroidal (or combinatorial) methods. In particular, the semistable reduction conjecture was reduced to a question about polyhedral subdivisions, see \cite[Conjecture 8.4]{AK}. Finally, toroidal methods were used to make $f'$ weakly semistable in the following sense: $B'$ is regular and $f'$ is flat and has geometrically reduced fibers, see \cite[Theorem~0.3 and \S0.8.2]{AK}. Thus, the gap from achieving semistable reduction was reduced to making $X'$ regular, but this turned out to be a difficult task. In \cite{Karu} Karu managed to solve it when $\dim(X)-\dim(B)\le 3$, and no further progress has been achieved until now. Nevertheless, the following aspects were improved in the meantime:

First, the language of toroidal geometry is being gradually replaced by a more robust language of log geometry, and we will use the latter in this paper. In particular, toroidal (resp. weakly semistable) morphisms were reinterpreted as log smooth (resp. and saturated) morphisms of log smooth log schemes. In addition, Molcho showed in \cite{Molcho} that one can use non-representable modifications (and toric stacks) instead of alterations, and this allows to proceed from the step of toroidal reduction in a canonical way.

Second, log smooth reduction was extended in \cite[Theorem~4.3.1]{tame-distillation} to a wide class of qe schemes at cost of allowing $X'\to(X\times_BB')^{\rm pr}$ to be an alteration whose degree is only divisible by primes non-invertible on $B$ (see also \cite[Section 3]{X}).

Third, if $B_0\subseteq B$ is open, $X_0=X\times_{B_0}B\to B_0$ is smooth, and the characteristic is zero, then log smooth reduction can be achieved by modifications $B'\to B$ and $X'\to X$ that restrict to isomorphisms over $B_0$ and $X_0$, see \cite{ATW-relative}.

\subsection{Main results}
Our main new ingredient is Theorem \ref{mainth} resolving the combinatorial conjecture of Abramovich and Karu. Lifting it to log schemes we obtain Theorem~\ref{mainlogth} about resolution of monoidal structure of morphisms of log schemes. In particular, it resolves log smooth morphisms to semistable ones, see Theorem~\ref{semistableth}. Similarly to \cite{Molcho}, we use root stacks and non-representable modifications $B'\to B$ instead of alterations to obtain the sharpest result. Using Kawamata trick one can then deduce the usual formulation with $B'\to B$ an alteration. Finally, we apply Theorem~\ref{semistableth} to improve the results on resolution of morphisms we have mentioned earlier. In particular, Theorem~\ref{mainapp} resolves a strong version of Conjecture~\ref{conj:AK} in characteristic zero, which also addresses divisors $Z\into X$ and controls the modification locus.

Finally, let us say a couple of words about our proof of Theorem \ref{mainth}. It can be viewed as a relative analogue of a difficult theorem of \cite{KKMS} on polyhedral subdivisions, so its resistance to attacks in the past is not so surprising. A breakthrough was obtained in the recent work \cite{ALPT} of the three authors and Pak, where a local version was established, and the current paper builds on the methods of \cite{ALPT} to obtain a complete solution. We would like to mention that a recent work \cite{Santos} of Haase, Paffenholz, Piechnik and Santos provided a relatively simple new proof of the subdivision theorem, and its constructions with Cayley polytopes were used as a starting point for the methods of \cite{ALPT} and this paper.

\section{Conical complexes}\label{conicalsec}

\subsection{The combinatorial conjecture of Abramovich and Karu}
In this subsection we recall the combinatorial semistable reduction conjecture from \cite{AK} in the slightly increased generality of maps with non-trivial vertical parts. In addition, it will be important that our complexes do not have to be connected.

\subsubsection{Terminology}
We follow the notation and terminology on rational cones, polyhedral complexes, etc., from \cite{AK}. Recall that a \emph{rational conical polyhedral complex}, or \emph{conical complex}, is an (abstract) polyhedral complex formed by gluing together finitely many polyhedral cells $\sigma$, each of which is equipped with a lattice $N_\sigma \cong \bfZ^{d_\sigma}$ such that $\sigma$ is a full-dimensional, rational, strictly convex polyhedral cone in $N_\sigma \otimes \bfR$, and such that if $\tau$ is a face of $\sigma$, then $N_\tau = N_\sigma|_{\Span(\tau)}$. We will denote such a conical complex by $\{(\sigma,N_\sigma)\}$. A \emph{map} $f \: X \to Y$ of conical complexes $X = \{(\sigma,N_\sigma)\}$ and $Y = \{(\rho,N_\rho)\}$ is a collection of group homomorphisms $f_\sigma \: N_\sigma \to N_\rho$, one for each $\sigma \in X$, such that the extension $f_\sigma \: N_\sigma \otimes \bfR \to N_\rho \otimes \bfR$ maps $\sigma$ into $\rho$, and such that if $\tau$ is a face of $\sigma$, then $f_\tau = f_\sigma|_\tau$. A conical complex is \emph{regular} (or nonsingular) if each of its cells $\sigma$ is generated by a lattice basis of $N_\sigma$. A subdivision $X'\to X$ with a regular $X'$ will be called a {\em resolution} of $X$.

\begin{rem}
The category of conical complexes is equivalent to the category of saturated Kato fans, in particular, it has fiber products. One can easily describe them explicitly, and we skip the details. In this section we prefer to work with cones because they are much more geometric, but we will switch to the language of fans in Section~\ref{ssrsec}.
\end{rem}

\begin{defin}[{\cite[Definition 8.1]{AK}}] \label{def:semistable}
A map of conical complexes $f \: X \to B$ is \emph{weakly semistable} if the following conditions hold.
\begin{enumerate}
\item For every $\sigma \in X$, we have $f(\sigma) \in B$.
\item For every $\sigma \in X$, we have $f(N_\sigma) = N_{f(\sigma)}$.
\item $B$ is regular.
\end{enumerate}
If, in addition, $X$ is regular, then $f$ is called \emph{semistable}.
\end{defin}

The condition $f^{-1}(0)=0$ in \cite{AK} means that the fibers are connected and have the trivial vertical part, so we do not include it.

\begin{defin}[{\cite[Definition 8.2]{AK}}]
Let $X$ be a conical complex. A \emph{lattice alteration} is a map $X' \to X$ where $X'$ is a conical complex of the form $\{ (\sigma,N'_\sigma) \: \sigma \in X \}$ and each $N'_\sigma$ is a sublattice of $N_\sigma$. An \emph{alteration} is a composition $X' \to X_1 \to X$ of a lattice alteration $X' \to X_1$ with a subdivision $X_1 \to X$. The alteration is \emph{projective} if the subdivision $X_1 \to X$ is projective. Finally, a subdivision is projective if there exists a piecewise linear function to the reals that is convex on each simplex of $X$ and whose domains of linearity are the faces of $X_1$ \cite{AK}.
\end{defin}

Fix a map $f \: X \to Y$ of conical complexes. For any alteration $g \: Y' \to Y$, there is an alteration $g'\:X' \to X$ with
\[
X' = \{ (\sigma \cap f^{-1}(g(\tau)), N_\sigma \cap f^{-1}(g(N_\tau)) \: \sigma \in X, \tau \in Y' \}
\]
which is the unique minimal alteration admitting a map $X' \to Y'$. It is easy to see that $X'=Y'\times_YX$ in the category of conical complexes, so we say that $g'$ is the {\em base change} of $g$ (with respect to $f$). If $g$ is projective, then $g'$ is projective as well.

We can now state the polyhedral conjecture of Abramovich and Karu.

\begin{conj} \label{conj:cones}
Let $f \: X \to B$ be a map of conical complexes. Then there exists a projective alteration $b\:B' \to B$ and a projective subdivision $a\:X'\to X\times_B B'$ such that $f'\:X' \to B'$ is semistable.
\end{conj}

\subsection{Statement of the main result}
Any $f'$ as above will be called a {\em resolution} of $f$, and we will construct a canonical resolution. This is convenient for the proof and important for applications, but should be spelled out carefully. Ideally, we would like the construction to be local, but the choice of $b$ must be affected by the whole fibers over cones of $B$. So, we propose the following technical solution.

\begin{defin}\label{qlocdef}
(i) A map $h\:Y\to X$ of conical complexes is called a {\em local isomorphism} if $f\:\sigma\toisom\rho$ and $f_\sigma\:N_\sigma\toisom N_\rho$ for any $\sigma\in Y$.

(ii) Loosely speaking, a construction is local if it is compatible with (or functorial with respect to) local isomorphisms. Constructions compatible with surjective local isomorphisms will be called {\em quasi-local}.
\end{defin}

\begin{exam}\label{localexam}
(i) Barycentric subdivision $X\mapsto X^b$ is the most basic local construction. It outputs a simplicial conical complex $X^b$. Moreover, the cones of $X^b$ are provided with a canonical {\em order}: the edges of each simplicial cone $\sigma$ are ordered by dimension of the cone they are mapped to in $X$. This induces an order on any further simplicial subdivision of $X$.

(ii) It was shown in \cite{KKMS} that any $X$ possesses a resolution $X^r$. It was later proved that the construction can be made local. For example, see \cite[Theorem~3.2.20]{Illusie-Temkin}, though there also are purely combinatorial constructions. We fix any such local (or just quasi-local) construction, since it will be used later.
\end{exam}

Now we can formulate a functorial strengthening of Conjecture~\ref{conj:cones}. Its proof will occupy Sections \ref{conicalsec} and \ref{polytopsec}.

\begin{theor}\label{mainth}
There exists a construction which associates to each map of conical complexes $f\:X\to B$ a resolution $b\:B'\to B$, $a\:X'\to X\times_BB'$, $f'\:X'\to B'$ in a quasi-local way: if $g\:Y\to C$ is another map and $\alpha\:Y\to X$, $\beta\:C\to B$ are surjective local isomorphisms such that $\beta\circ g=f\circ\alpha$, then the resolution $g'$ of $g$ is the base change of $f'$ in the sense that $C'=C\times_BB'$ and $Y'=Y\times_XX'$.
\end{theor}

\subsection{Some reductions}\label{sec:red}
As often happens with resolution problems, the solution will be constructed by composing few steps that gradually improve $f$. We start with a few standard reductions. (Their analogues )

\subsubsection{Quasi-localization}
Now, we are going to essentially use non-connected complexes and a typical descent argument. For a conical complex $X$ let $X_0$ denote the disjoint union of cones of $X$ and let $X_1=X_0\times_XX_0$, in particular, $X_0\to X$ and both projections $X_1\toto X_0$ are surjective local isomorphisms. Any quasi-local resolution of $f\:X\to B$ is compatible with the resolution of $f_0\:X_0\to B_0$, and hence is determined by the latter. Moreover, any quasi-local resolution on the category $\calC_0$ of disjoint unions of cones automatically extends to the category $\calC$ of all conical complexes. Indeed, $f_1\:X_1\to B_1$ also lies in $\calC_0$, hence both pullbacks to $f_1$ of the resolution $f'_0$ of $f_0$ are isomorphic to the resolution $f'_1$ of $f_1$ by the quasi-locality. This implies that $f'_0$ descends to a resolution $f'$ of $f$.

To summarize, it suffices to prove Theorem~\ref{mainth} in the case when $B$ and $X$ are disjoint unions of cones. We cannot work all the time within the category $\calC_0$ because it is not closed under subdivisions. However, each (quasi-local) step of the construction can be made under the assumption that the input is in $\calC_0$. We will use this additional localization a couple of times during the initial reduction process and each time it will be mentioned explicitly.

\subsubsection{Eliminating vertical components}
First step of our construction of $f'$ is just to resolve the source. To simplify the notation we replace $f$ by $f^r\:X^r\to B$, where $X^r$ is the resolution from Example~\ref{localexam}, and assume now that $X$ is regular. By quasi-localization, we can assume further that $X$ is a disjoint union of regular simplicial cones $X_i$. Let $X_i=X_i^v\times X_i^h$, where the vertical part $X_i^v$ is mapped by $f$ to a point and $f$ is injective on the edges of the horizontal part $X_i^h$. It suffices to construct a resolution $\coprod (X_i^h)'\to B'$ of $\coprod X_i^h\to B$ because then $\coprod (X_i^h)'\times X_i^v\to B'$ will be a resolution of $f$. Absence of vertical components is preserved by subdivisions and base changes, so we can assume in the sequel that it holds for $f$. (Notice that the regularity of $X$ will be destroyed by the next step.)

\subsubsection{Ordering the complexes}
Replace $B$ by the barycentric subdivision $B^b$ and update $X$ via the base change $X\times_B B^b$. By \ref{localexam}(i), from this stage the simplices of $B$ and any its subdivision acquire a canonical ordering. In the same way, replacing $X$ by $X^b$, we provide $X$ and its further subdivisions with an order.

\subsubsection{Regularizing the base}
As a next step, we replace $B$ by its resolution $B^r$ and update $X$ via the base change $X\times_B B^r$. By the quasi-localization we can further assume that $B=\coprod B_i$, where $B_i$ are regular simplicial cones. From this point on, any further alteration of $B$ will be of the form $B_c\to B$, where $c\ge 1$, one replaces each $N_i$ by $cN_i$, and subdivides each $B_i$ into $(\dim B_i)^c$ regular cones in a standard way, see Theorem~\ref{thm:cansimp} below. (We perform both operations together for convenience of the exposition -- at some stages one of them would suffice.) In particular, the assumption that $B$ is regular will be maintained until the end of the proof.

\subsubsection{Making $f$ weakly semistable}
If $c$ is such that all edges of $X$ are mapped to edges of $B_c$, then replacing $B$ by $B_c$ and replacing $f$ by the base change $f_c\:X\times_BB_c\to B_c$ we achieve that $f(\sigma)$ is a cone for any $\sigma\in X$. So, choosing the minimal such $c$, we obtain a quasi-local construction which improves $f$ so that $B$ is regular and $f$ takes cones to cones. Next, choose the minimal $c$ such that $cN_{f(\sigma)}\subseteq f(N_\sigma)$. Replacing $f$ by $f_c$ we also achieve that $N_{f(\sigma)}=f(N_\sigma)$, and hence $f$ becomes weakly semistable.

The weak semistability might be destroyed by subdivisions of $X$, but after each such step we will perform the above step to restore it. That is, we will simply replace $f$ by the weakly semistable $f_c$ with the minimal possible $c$.

\subsubsection{Reduction to polytopes}
Finally, as in \cite{KKMS} it will be convenient to perform deeper geometric constructions on the level of lattice polytopes. We will do so in detail in the next section: Each component $B_i$ of $B$ has a canonical section $\Delta_i$ passing through the minimal lattice points on the edges. Since there are no vertical components, the preimage $Z$ of $\Delta=\coprod_i\Delta_i$ in $X$ has a natural structure of a disjoint union of lattice polytopes. As in \cite{KKMS}, the problem now reduces to finding an appropriate $c$ and subdividing the preimage $Z(c)$ of $c\Delta$ into regular lattice simplices. The precise formulation and the construction are in Section~\ref{polytopsec}.

\section{Proof of the main theorem}\label{polytopsec}
A \emph{lattice polytope} is a pair $(P,N_P)$ (usually denoted as just $P$) where $N_P \cong \bfZ^{d_P}$ is an affine lattice and $P$ is a polytope in $N_P \otimes \bfR$ with vertices in $N_P$. A \emph{face} of a lattice polytope $(P,N_P)$ is a lattice polytope $(F,N_F)$ where $F$ is a face of $P$ and $N_F = N_P |_{\AffSpan(F)}$. A \emph{lattice polytopal complex}, or \emph{polytopal complex}, is a polyhedral complex formed by gluing together finitely many lattice polytopes (with the usual gluing conditions of polyhedral complexes) such that lattices agree on common faces and each polytope $P$ is full-dimensional in $N_P$. Maps between polytopal complexes are defined analogously to the conical case.

Let $(P,N_P)$ be a lattice polytope. We let $L_P$ denote the affine sublattice of $N_P$ spanned by the vertices of $P$. We say that the \emph{index} of $P$ is the index
\[
[N_P \cap \AffSpan(P) : L_P].
\]
A \emph{unimodular simplex} is a lattice simplex of index 1. A \emph{unimodular triangulation} is a polytopal complex all of whose elements are unimodular simplices. If the vertices of $P$ are contained in some sublattice $N'$ of $N_P$, then we say that the \emph{index of $P$ with respect to $N'$} is the index of $(P,N')$. We define unimodularity with respect to $N'$ accordingly.

Usually, we will specify an origin in $N_P$, which allows us to define dilations of lattice polytopes and Minkowski sums of polytopes with the same lattice. Since all of our results will be invariant under translation, the choice of origin will not matter. If $X = \{(P,N_P)\}$ is a polytopal complex, then we can define a polytopal complex $cX = \{(cP,N_{P})\}$. Given a map $f : X \to Y$ of polytopal complexes, there is the obvious induced map $cX \to cY$; we denote this map by $cf$.

We define lattice alterations and alterations of polytopal complexes analogously to the conical case. For induced alterations, we need to proceed more carefully. Let $f : X \to Y$ be a map of polytopal complexes and $g : Y_1 \to Y$ an alteration. If the rational subdivision
\begin{equation} \label{eq:induced}
X_1 := \{ (P \cap f^{-1}(g(Q)), N_P \cap f^{-1}(g(N_Q))) : P \in X, Q \in Y_1 \}
\end{equation}
is a lattice polytopal complex, then $X_1 \to X$ is the unique minimal alteration admitting a map $X_1 \to Y_1$, and we say that $X_1 \to X$ is \emph{induced} by $g$ (with respect to $f$).

We say a map $f : X \to Y$ of polytopal complexes is \emph{good} if for every $P \in X$, we have $f(P) \in Y$. We have the following.

\begin{prop} \label{prop:good}
Let $f : X \to Y$ be a map of polytopal complexes. Then there exists a positive integer $c$ and a projective subdivision $Y' \to cY$ which induces with respect to $cf$ an alteration $X' \to cX$ such that $X' \to Y'$ is good.
\end{prop}

\begin{proof}
It is easy to see that there is a rational projective subdivision $\tilde{Y}$ of $Y$ such that for every $P \in X$, we have that $f(X)$ is a union of cells of $Y$. This induces as in \eqref{eq:induced} a rational subdivision $\tilde{X}$ of $X$. For some $c$, we have that $c\tilde{X}$ and $c\tilde{Y}$ are lattice subdivisions of $cX$ and $cY$, and the map $c\tilde{X} \to c\tilde{Y}$ gives the result.
\end{proof}

\begin{prop} \label{prop:induced}
Let $f : X \to Y$ be a good map of polytopal complexes and $Y_1 \to Y$ an alteration. Then there exists a positive integer $c$ such that if $Y_1 \to cY_1$ is the lattice alteration given by $(P,N_P) \cong (cP,cN_P) \mapsto (cP,N_P)$ and $g$ is the alteration $Y_1 \to cY_1 \to cY$, then $g$ induces with respect to $cf$ an alteration $X_1 \to cX$.
\end{prop}

\begin{proof}
As in the previous proof, there is $c$ such that $cY_1 \to cY$ induces a lattice subdivision $X' \to cX$. Then the lattice alteration $Y_1 \to cY_1$ induces an alteration $X_1 \to X'$, and $X_1 \to X' \to cX$ is the desired alteration.
\end{proof}

Our goal now is to prove the following theorem:

\begin{theor} \label{thm:polytopes}
Let $f : X \to B$ be a map of polytopal complexes. Then there exists a positive integer $c$, a projective alteration $B_1 \to cB$ which induces with respect to $cf$ an alteration $X_1 \to cX$, and a projective subdivision $Y \to X_1$ such that $Y$ and $B_1$ are both unimodular triangulations. The construction is quasi-local.
\end{theor}

The quasi-locality falls of naturally as a product of the proof, relying on local improvements of the index we detail in the following. Additionally, observe that the proof of the semistable reduction conjecture, and more generally Theorem~\ref{mainth} follows as discussed in Section~\ref{sec:red}.
Indeed, as observed, it even suffices to use a particular case of Theorem~\ref{thm:polytopes}, where $X$ and $B$ are ordered, $B$ is a disjoint union of unimodular simplices, etc. We stated it here for its combinatorial beauty in full generality nevertheless.

%

\subsection{Canonical subdivisions}

As we will see, much of our proof relies on being able to construct ``canonical'' subdivisions for polytopes. We now formalize his notion.

An \emph{ordered polytope} is a polytope along with a total order on its vertices. A \emph{face} of an ordered polytope is a face of the underlying polytope along with the induced ordering. Let $\calP$ to be the category whose objects are ordered lattice polytopes and whose morphisms are $F \to P$ where $F$ is a face of $P$. Let $\calS$ be the category whose objects are ordered subdivisions of ordered lattice polytopes and whose morphisms are $F' \to P'$ where $F'$ is the subdivision induced on a face of the underlying ordered polytope of $P'$.

Let $\Gamma : \calF \to \calP$ be a full and faithful functor for some category $\calF$. A \emph{canonical subdivision} of $\Gamma$ is a functor $\Sigma: \calF \to \calS$ such that $\Sigma(P)$ is a subdivision of $\Gamma(P)$ for all $P \in \Ob(\calF)$. If $\Sigma(P)$ is a triangulation for all $P$, then we call $\Sigma$ a \emph{canonical triangulation}. If $\Sigma(P)$ is projective for all $P$, then we say that $\Sigma$ is \emph{projective}.

\begin{rem}
In what follows, we will generally omit justification for why certain canonical subdivisions are projective. This is standard practice, but we should note here how one would go about proving projectivity. Our subdivisions will be built up inductively, first constructing an initial subdivision, then subdividing each cell of this subdivision, and so on. At each step, each intermediate subdivision of a cell can easily be shown to be projective. However, this does not guarantee projectivity of the entire subdivision. To prove this, one should demonstrate that at each step, the intermediate subdivisions used are not only canonical (in the sense described above), but also that one can choose \emph{height functions} to induce these subdivisions as projective subdivisions in a canonical way. In other words, one should define a category $\calS_{\text{proj}}$ whose objects are pairs $(P,f)$, where $P$ is an ordered polytope and $f$ is a height function on the vertices of $P$, and whose morphisms are $(F,f|_F) \to (P,f)$ where $F$ is a face of $P$. A projective canonical subdivision of $\Gamma$ (where $\Gamma$ is defined above) is then defined to be a functor $\Sigma : \calF \to \calS_{\text{proj}}$ such that for all $P \in \Ob(\calF)$, we have $\Sigma(P) = (\Gamma(P),f)$ for some $f$. Each projective canonical subdivision gives a canonical subdivision by replacing $(\Gamma(P),f)$ with the subdivision induced by $f$. We leave it to the reader to show that all subdivisions described later can be given as projective canonical subdivisions as defined here.
\end{rem}

\subsubsection{Canonical triangulations of dilated simplices}

Let $\Delta$ be the category defined as follows. The objects are ordered pairs $(P,c)$ where $P$ is an ordered lattice simplex and $c$ is a positive integer. The morphisms are $(P',c) \to (P,c)$ where $P'$ is a face of $P$. We have a full and faithful functor $\mu : \Delta \to \calP$ defined by $\mu(P,c) = cP$. We will assume the origin is in $L_P$, so that $cP$ has vertices in $L_P$.

The following is a key result from Haase et al.\ \cite{Santos}.

\begin{theor} \label{thm:cansimp}
There is a projective canonical triangulation $\Sigma$ of $\mu$ such that for all $(P,c) \in \Ob(\Delta)$, we have that $\Sigma(P,c)$ is unimodular with respect to $L_P$.
\end{theor}

Later we will prove a generalization of this to polysimplices, Lemma \ref{lem:polysubdv}. For now, we will state a modified version of this Theorem. Let $\Delta'$ be the full subcategory of $\Delta$ whose objects are $(P,c) \in \Ob(\Delta)$ with $c \ge \dim(P) + 1$. Let $\mu'$ be the restriction of $\mu$ to $\Delta'$. Then we have the following.

\begin{lem} \label{lem:stcan}
There is a projective canonical triangulation $\Sigma'$ of $\mu'$ such that for all $(P,c) \in \Ob(\Delta')$ and all full-dimensional simplices $Q$ of $\Sigma'(P,c)$, we have the following:
\begin{enumerate}
\item $L_P = L_Q$. (That is, $\Sigma'(P,c)$ is unimodular with respect to $L_P$.)
\item If the vertices of $Q$ are ordered $v_1$, $v_2$, ~\dots, $v_{\dim(P)+1}$, then for all $i = 1$, ~\dots, $\dim(P)$, every face of $P$ which contains $v_i$ also contains $v_{i+1}$.
\end{enumerate}
\end{lem}

\begin{proof}
Note: This proof uses ideas and notation from the next section. We have put the proof in this section for the sake of organization.

We construct the triangulation $\Sigma'(P,c)$ of $cP$ as follows. For each face $F$ of $P$, let $O_F$ be the barycenter of $F$. If $F$ has dimension $k$ and $F'$ is a face of $F$, then we note that
\[
\phi(F',F) := (c-k-1)F' + (k+1)O_F
\]
is a lattice polytope contained in $cF$.

Let $d = \dim(P)$. Let $F_{r}$, $F_{r+1}$, ~\dots, $F_d$ be a sequence of nonempty faces of $P$ with $F_r < ~\dotsb < F_d$ and $\dim(F_i) = i$ for all $r \le i \le d$. We define
\[
(cP)_{F_r,\dots,F_d} := \conv \bigcup_{i=r}^d \phi(F_{r},F_i).
\]
Then the collection of all such $(cP)_{F_{r},\dots,F_d}$ are the full-dimensional cells of a subdivision $\Sigma$ of $cP$.

The final step is to refine $\Sigma$ to a triangulation. When viewed as lattice polytopes in $L_P$, each $(cP)_{F_r,\dots,F_d}$ is lattice equivalent to the Cayley polytope
\[
\Cay( (c-r-1)F_r, (c-r-2)F_r, \dots, (c-d-1)F_r ).
\]
Thus, by Lemma~\ref{lem:polysubdv}, there are canonical triangulations of each $(cP)_{F_r,\dots,F_d}$ which are unimodular with respect to $L_P$. These extend to a triangulation of $\Sigma$ which is unimodular in $L_P$. The fact that this triangulation satisfies property (2) is easy to check, as is canonicity.
\end{proof}

\subsection{Cayley polytopes}

Let $(P_1,\bfZ^d)$, $(P_2,\bfZ^d)$ ~\dots, $(P_n,\bfZ^d)$ be lattice polytopes with the same lattice $\bfZ^d$. To simplify some statements later, we will assume that $L_{P_j}$ contains the origin for all $j$. (Note that we can always translate the $P_j$ so that this holds.) Let $P$ be the column array $(P_1 \dots P_n)^T$. We define the \emph{Cayley polytope} $\Cay(P)$ to be the polytope
\[
\Cay(P) := \conv \left( \bigcup_{i=1}^n P_i \times e_i \right) \subset \bfR^d \times \bfR^n
\]
where $\conv$ denotes convex hull and $e_i$ is the $i$-th standard basis vector of $\bfR^n$. We will also occasionally write $\Cay(P_1,\dots,P_n)$ instead of $\Cay(P)$.

We make $\Cay(P)$ a lattice polytope by equipping it with the affine lattice $\bfZ^d \times \Lambda^{n-1}$, where $\Lambda^{n-1} := \{ x \in \bfZ^n : x_1 + \dots + x_n = 1 \} \cong \bfZ^{n-1}$. We have
\[
L_{\Cay(P)} = L_P \times \Lambda^{n-1}
\]
where $L_P := \langle L_{P_1}, \dots, L_{P_n} \rangle$ is the lattice generated by $L_{P_1}$, \dots, $L_{P_n}$. We define
\[
\ind(P) := [ \bfZ^d \cap (L_P \otimes \bfR) : L_P],
\]
so that $\ind(\Cay(P)) = \ind(P)$.

If $P_1$, ~\dots, $P_n$ are ordered polytopes, then we make $\Cay(P)$ an ordered polytope with the following ordering: First the vertices of $P_1 \times e_1$ in the order given by $P_1$, then the vertices of $P_2 \times e_2$ in the order given by $P_2$, and so on. (Note that these are precisely the vertices of $\Cay(P)$.)

Let $A$ be an $m \times n$ matrix with nonnegative integer entries and let $P$ be as above. We define
\[
AP := \left( \sum_{j=1}^n A_{1j}P_j, \sum_{j=1}^n A_{2j}P_j, \dots, \sum_{j=1}^n A_{mj}P_j \right)^T,
\]
with the sum being Minkowski sum. In other words, we define matrix multiplication in the expected way. If $A$ is a row vector, then $AP$ has one entry, and we identify $AP$ with this entry.

A \emph{polysimplex}, or \emph{product of simplices}, is a polytope of the form $\sum_j P_j$, where $\{P_j\}$ is an affinely independent set of simplices. If there is an ordering on such a set $\{P_j\}$ and each $P_j$ is an ordered polytope, then we make $\sum_j P_j$ an ordered polytope by lexicographic ordering on its vertices. Thus, if $A$ is a matrix as above and $P = (P_1 \dots P_n)^T$ with $\{P_j\}_{j=1}^n$ an affinely independent set of ordered simplices, then each entry of $AP$ is an ordered polysimplex, and $\Cay(AP)$ is an ordered polytope. In addition, with the assumption that $L_{P_j}$ contains the origin for all $j$, the vertices of $\Cay(AP)$ are contained in $L_{P} \times \Lambda^{m-1}$.

We set some final notation regarding matrices. In the following, assume $P$ is a column array with $n$ lattice polytope entries and $A$ is an $m \times n$ nonnegative integer matrix. Let $A_i$ denote the $i$-th row of $A$. Let $\supp A$ denote the set of column indices at which $A$ is not a 0-column. Let $P[ i, q]$ denote the column matrix obtained by replacing the $i$-the entry of $P$ with $q$. Let $P_A$ denote the column matrix obtained by restricting $P$ to the entries indexed by $\supp A$. If $\supp A = \emptyset$, then we define $P_A$ to be a single entry which is the origin.

\subsection{Overview of the proof}

In this section, we will reduce the problem to constructing certain canonical triangulations of Cayley polytopes of polysimplices.

Let $\calF$ be the category whose objects are tuples $(P,A)$ satisfying the following.
\begin{enumerate}
\item $P = (P_1,\dots,P_n)^T$  where $P_1$, \dots, $P_n$ are affinely independent ordered lattice simplices with the same lattice.
\item $A$ is an $m \times n$ matrix with nonnegative integer entries.
\end{enumerate}
The morphisms in $\calF$ are $(F,A') \to (P,A)$, where $F_i$ is a face of $P_i$ for all $i$, and $A'$ is obtained from $A$ by taking a subset of the rows of $A$. We have a full and faithful functor $\Cay : \calF \to \calP$ given by $(P,A) \mapsto \Cay(AP)$.

\begin{lem} \label{lem:polysubdv}
There is a projective canonical subdivision $\Sigma$ of $\Cay$ such that for all $(P,A) \in \Ob(\calF)$, the triangulation $\Sigma(P,A)$ is unimodular with respect to $L_P \times \Lambda^{m-1}$.
\end{lem}

We postpone the proof to the next section, and proceed to find a way to lower the indices of the polytopes in a polytopal complex. For this, we will use the idea of ``box points'' (or Waterman points) as introduced by Waterman for the KMW theorem. (Note that our final construction will be different from theirs when restricted to simplices.)

Let $P = (P_1 \dots P_n)^T$ where $\{P_j\}$ is an affinely independent set of lattice simplices in $\bfZ^d$. For simplicity, assume that $L_{P_j}$ contains the origin for all $j$. Recall that $L_P$ is the lattice generated by $L_{P_1}$, \dots, $L_{P_n}$. A \emph{box point} of $P$ is a nonzero element of $G_P := (\bfZ^d \cap (L_P \otimes \bfR)) / L_P$. If $F = (F_1 \dots F_n)^T$ is such that $F_j$ is a face of $P_j$ for all $j$, then there is a natural inclusion $G_F \into G_P$, and so any box point of $F$ can be regarded as a box point of $P$. Moreover, if we have two arrays $P$, $Q$ with $L_P = L_Q$, then we identify the box points of $P$ with the box points of $Q$.

We make the following crucial observation: If $\bfm$ is a box point of $P$, then there is a unique minimal (in the product order) $N$-tuple $(c_1,\dots,c_n)$ of nonnegative integers such that the polytope $\sum_{j=1}^n c_j P_j$ contains a representative of $\bfm$. This tuple satisfies $0 \le c_j \le \dim P_j$ for all $j$, and the representative is unique. We denote the row vector $(c_1 \dots c_n)$ by $c(P,\bfm)$. If $\bfm$ is not a box point of $P$, then we set $c(P,\bfm)$ to be the 0-vector. Note that if $F$ is as above and $\mathbf{m}$ is also a box point of $F$, then $c(F,\bfm) = c(P,\bfm)$.

Let $\bfm$ be a box point of $P_0$ for some $(P_0,A_0) \in \Ob(\calF)$. We define $\calF^{\bfm}$ to be the full subcategory of $\calF$ whose objects are $(P,A) \in \Ob(\calF)$ satisfying the following:
\begin{enumerate}
\item Let $c = c(P, \bfm)$. Then for all $i \in [m]$ and $j \in [n]$, we have
\[
A_{ij} = 0 \text{ or } A_{ij} \ge c_j.
\]
\item We have
\[
\supp A_1 \cap \supp c \supseteq \supp A_2 \cap \supp c \supseteq \dots \supseteq \supp A_m \cap \supp c.
\]
\end{enumerate}
Let $\Cay^{\bfm}$ be the restriction of $\Cay$ to $\calF^{\bfm}$.

We will prove the following in Section~\ref{sec:boxsubdv}

\begin{lem} \label{lem:prebox}
Let $\bfm$ be as above. Then there is a projective canonical triangulation $\Sigma^{\bfm}$ of $\Cay^{\bfm}$ such that for all $(P,A) \in \Ob(\calF^{\bfm})$ and all full-dimensional simplices $Q$ in $\Sigma^{\bfm}(P,A)$, we have the following.
\begin{itemize}
\item If $\bfm$ is a box point of $P_A$, then
\[
\ind(Q) < \ind(P_A).
\]
\item If $\bfm$ is not a box point of $P_A$, then $L_Q = L_{P_A} \times \Lambda^{m-1}$.
\end{itemize}
\end{lem}

Assuming this lemma, we can now prove Theorem~\ref{thm:polytopes}.

\begin{proof}[Proof of Theorem~\ref{thm:polytopes}]
By Proposition~\ref{prop:good}, we may assume $f : X \to Y$ is a good map.

\emph{Step 1: Reducing the base.}

As discussed before, the first step is to alter $B$ so that it is a unimodular triangulation. By the KMW theorem \cite{KKMS}, there is a positive integer $c$ so that we have a projective unimodular triangulation $B' \to cB$. Now, by Proposition~\ref{prop:induced}, there is a positive integer $c'$ and an alteration $B' \to c'cB$ such that this alteration induces with respect to $c'cf$ an alteration $X_1 \to c'cX$ and map $X_1 \to B'$. Hence, we may assume $B$ is a unimodular triangulation. Choosing any triangulation, for example, the barycentric one, we may assume that $X$ is triangulated.



\emph{Step 2: Lowering the index.}

By Lemma~\ref{lem:stcan}, there exists $c$ such that for each $Q \in B$ we have a unimodular projective triangulation
\[
\Sigma'(Q,c) \to cQ.
\]
Since these triangulations are canonical, this gives a unimodular triangulation $B_1$ of $cB$. 

By Proposition~\ref{prop:induced}, for some $c'$, the alteration $B_1 \to c'B_1 \to c'cB$ induces an alteration $X_1 \to c'cX$ and a map $f_1 : X_1 \to B_1$. We may assume $c' > \dim X$.

Suppose $Q \in B$ and $\Delta$ is a full-dimensional simplex of the complex $f^{-1}(Q)$. Let $v_1$, ~\dots, $v_n$ be the vertices of $Q$, and let
\[
P_i := (f^{-1}(v_i) \cap \Delta, N_{f^{-1}(v_i)})
\]
for all $i$. Let $P = (P_1 \dots P_n)^T$. Note that the $P_i$ are affinely independent simplices and $\Delta = \Cay(P)$. 

Let $v$ be a vertex of $B_1$ contained in $cQ$. Define
\[
P(v) := f^{-1}_1(v) \cap c'c\Delta.
\]
Let $(a_1,\dots,a_n)$ be the barycentric coordinates of $v$ with respect to the vertices $cv_1$, ~\dots, $cv_n$ of $cQ$. Since $Q$ is unimodular, $ca_1$, ~\dots, $ca_n$ are nonnegative integers. From the definition of $X_1$, we have
\begin{align*}
P(v) &= a_1 P(cv_1) + a_2 P(cv_2) + \dotsb + a_kP(cv_n) \\
&= c'c (a_1 P_1 + a_2P_2 + \dotsb + a_nP_n).
\end{align*}
Thus, if $Q_1 \in B_1$ has vertices $u_1$, ~\dots, $u_m$, we have
\begin{align*}
f_1^{-1}(Q_1) \cap c'c\Delta &= \Cay( P(u_1), \dots, P(u_m) ) \\
&= \Cay( A P )
\end{align*}
where $A$ is an $m \times n$ matrix of nonnegative integers divisible by $c'$. By Lemma~\ref{lem:stcan}, we also have
\[
\supp A_1 \supseteq \supp A_2 \supseteq \dots \supseteq \supp A_m.
\]
This argument shows that every $R \in X_1$ is of the form $\Cay(A_R P_R)$ with $(P_R,A_R)$ satisfying the above conditions and where $\Delta_R := \Cay(P_R)$ is an element of $X$.


Let $\bfm$ be a box point of $P$. By the above conditions on $A_R$, we have that $(P_R,A_R) \in \Ob(\calF^{\bfm})$ for all $R \in X_1$.
Thus, by Lemma~\ref{lem:prebox}, we have a projective triangulation $Y \to X_1$ where each $R \in X_1$ is triangulated into $\Sigma^{\bfm}(P_R,A_R)$. If $\bfm$ is a box point of $P_R$, then for every full-dimensional simplex $Q$ of $Y|_R$ we have $\ind(Q) < \ind(P_R) = \ind(\Delta_R)$, and otherwise $L_Q = L_{\Delta_R}$.

Now repeat the process of Step 2 with $Y$ instead of $X$. Each time we do this procedure, we lower the indices of some of the lattices spanned by elements of $X$ while keeping the other lattices the same. Eventually all lattices will be unimodular, completing the proof.
\end{proof}

\subsection{Canonical triangulations of Cayley polytopes} \label{sec:polysubdv}
To prove Lemma~\ref{lem:prebox}, we will first prove Lemma~\ref{lem:polysubdv}. We will give two distinct triangulations satisfying this lemma. Both are generalizations of the Haase et al.\ construction \cite{Santos}, and both are the same when restricted to polysimplices. The first triangulation is described in the authors' previous work \cite{ALPT} with Pak. We will use this triangulation mainly to help define the second triangulation, which is what we will need later. The definitions given here are entirely recursive; for a more explicit description of the first triangulation, see \cite{ALPT}. The advantage of the recursive approach is that it will make very abstract statements easier to prove.

\subsubsection{The first triangulation} \label{sec:polysubdv1}

Let $(P,A) \in \calF$. Given indices $i$, $j$, we say that $A$ is \emph{$(i,j)$-reducible} if $A_{ij} > 0$.
Suppose $A$ is $(i,j)$-reducible, and let $\phi : \Cay(AP) \to Q$ be any affine map. We will define a set $f_{ij}(P,A,\phi)$ as follows.

Let $A'$ be the matrix obtained from $A$ by subtracting 1 from the $A_{ij}$ entry. If $\dim P_j = 0$, then we set
\[
f_{ij}(P,A,\phi) = \{(P,A',\phi)\}.
\]
Note that in this case $\Cay(AP) = \Cay(A'P)$, since by the assumptions in the previous section $P_j$ is the origin. We therefore have a map $\phi: \Cay(A'P) \to Q$.

Now assume $\dim P_j > 0$. First assume $\sum_{i'} A_{i'j} > 1$. Let $v$ be the first vertex of $P_j$ and let $F_j$ be the facet of $P_j$ opposite $v$. Let $F$ be the array obtained from $P$ by replacing $P_j$ with $F_j$. Then $\Cay(AP)$ has a subdivision into the two full-dimensional polytopes, one of which is
\[
Q_1 := \Cay( (AP)[i, v + A'_i P])
\]
and the other of which is
\[
Q_2 := \Cay( (AF)[i, \conv( (v + A'_i F) \cup (A_i F))]).
\]
We have a lattice polytope isomorphism $\phi_1 : \Cay(A'P) \to Q_1$. In addition, we have an affine isomorphism $\phi_2 : \Cay(A''F) \to Q_2$, where $A''$ is obtained from $A$ by inserting the row $A_i'$ above the $i$-th row of $A$. Note that $Q_1$ and $Q_2$ have vertices in $L_P \times \Lambda^{m-1}$. If we view $Q_2$ as having the ambient lattice $L_P \times \Lambda^{m-1}$, then $\phi_2$ is a lattice polytope isomorphism.

In the above situation, we define
\[
f_{ij}(P,A,\phi) := \{ (P,A',\phi \circ \phi_1), (F,A'',\phi \circ \phi_2) \}
\]
We order $f_{ij}(P,A,\phi)$ so that $(P,A',\phi \circ \phi_1)$ the first element and $(F,A'',\phi \circ \phi_2)$ is the second. Note that we have maps $\phi \circ \phi_1 : \Cay(A'P) \to Q$ and $\phi \circ \phi_2 : \Cay(A''F) \to Q$.

Finally, assume that $\sum_{i'} A_{i'j} = 1$. Let $P^\square$ be the array obtained from $P$ by replacing $P_j$ with 0, and let $A^\square$ be the matrix obtained from $A$ by replacing the $i$-th row of $A$ with $\dim P_j + 1$ copies of $A_i'$. Then we have an affine isomorphism $\phi^\square : \Cay(A^\square P^\square) \to \Cay(AP)$. This is a lattice polytope isomorphism if $\Cay(AP)$ is viewed in the lattice $L_P \times \Lambda^{m-1}$. We set
\[
f_{ij}(P,A,\phi) := \{ (P^\square, A^\square, \phi \circ \phi^\square) \}.
\]
Again, we have a map $\phi \circ \phi^\square : \Cay(A^\square P^\square) \to Q$.

Now, for $(P,A) \in \calF$, we say that $A$ is \emph{$j$-reducible} if $A_{ij} > 0$ for some $i$. If $A$ is $j$-reducible, we define
\[
f_j(P,A,\phi) := f_{i_0 j}(P,A,\phi),
\]
where
\[
i_0 := \min \{ i : A_{ij} = \max_{i'} A_{i'j} \}.
\]
This is well-defined because by definition, $A_{i_0 j} > 0$.

We now define a process as follows. Fix $(P_0,A_0) \in \calF$. We keep track of a set $\calS$, which is initially $\{ (P_0,A_0,\mathrm{id}) \}$, where $\mathrm{id}$ is the identity map. A \emph{move} consists of the following. First, we choose and element $\alpha = (P,A,\phi)$ of $\calS$ and a $j$ such that $A$ is $j$-reducible. We then modify $\calS$ by replacing $\alpha$ with the elements of $f_j(\alpha)$. We will call this a ``$j$-move applied to $\alpha$''. We apply such moves one-by-one until there are no moves available.

At each step of the process, we have that
\[
\Phi(\calS) := \{ \phi(\Cay(AP)) : (P,A,\phi) \in \calS \}
\]
is a polytopal dissection of $\Cay(P_0 A_0)$; i.e., it is a set of full-dimensional polytopes with pairwise disjoint interiors and whose union is $\Cay(P_0 A_0)$. In fact, we have the following.

\begin{lem} \label{lem:polysubdv1}
The above process always terminates and the final result $\calS_{\mathrm{final}}$ is independent of the moves chosen. Moreover, the map
\[
(P_0,A_0) \mapsto \Phi(\calS_{\mathrm{final}})
\]
is a projective canonical triangulation $\Sigma$ of $\Cay$ such that $\Sigma(P_0,A_0)$ is unimodular with respect to $L_{P_0} \times \Lambda^{m-1}$.
\end{lem}

\begin{proof}
For $(P,A,\phi) \in \calS$, a consideration of how the dimensions of the entries of $P$ and the volume of $\Cay(AP)$ change after each move easily proves termination of the process.

We now prove uniqueness of $\calS_{\mathrm{final}}$. By the diamond lemma, it suffices to prove the following: Suppose $\calS$ is an intermediate configuration and $\calS_1$ and $\calS_2$ are the results of applying moves $M_1$ and $M_2$, respectively, to $\calS$. Then there is a sequence of moves starting from $\calS_1$ and a sequence of moves starting from $\calS_2$ which both end in the same configuration.

If $M_1$ and $M_2$ apply to different elements of $\calS$, then they clearly commute, and we are done. Assume they apply to the same element $\alpha = (P,A,\phi) \in \calS$. We can assume $M_1$ and $M_2$ replace $\alpha$ with $f_{i_1j_1}(\alpha)$ and $f_{i_2 j_2}(\alpha)$, respectively, with $j_1 \neq j_2$. First assume $i_1 \neq i_2$. Then applying a $j_2$-move to all elements of $f_{i_1j_1}(\alpha)$ yields the same result as applying a $j_1$-move to all elements of $f_{i_2 j_2}(\alpha)$, as desired. (Note that these moves are legal due to the existence of $M_1$ and $M_2$.)

Now assume $i_1 = i_2 = i$. First suppose that $\sum_{i'} A_{i'j_1} > 1$ and $\sum_{i'} A_{i'j_2} > 1$. Apply a $j_2$-move to all elements of $f_{i j_1}(\alpha)$. Then, if $f_{i j_1}(\alpha)$ has two elements, apply a $j_2$-move to $f_{j_2}(f_{i j_1}(\alpha)_2)_1$, where $S_i$ denotes the $i$-th element of an ordered set $S$. We can check that this process gives the same result if we replace the roles of $j_1$ and $j_2$, as desired.

Next suppose $\sum_{i'} A_{i'j_1} = 1$ and $\sum_{i'} A_{i'j_2} > 1$. We first define a sequence of moves starting from $\calS_1$. First, apply a $j_2$-move to the one element of $f_{i j_1}(\alpha)$. Then apply a $j_2$-move to $f_{j_2}(f_{i j_1}(\alpha)_1)_1$. Then apply a $j_2$-move $f_{j_2}(f_{j_2}(f_{i j_1}(\alpha)_1)_1)_1$, and so on, until we have applied a total of $\dim P_{j_i} + 1$ moves. Now we define a sequence of moves starting from $\calS_2$. Starting from the set $\{f_{ij_2}(\alpha)\}$, apply $j_1$-moves to elements until there are no $j_1$-moves left. We again check that these two processes have the same result, as desired.

Finally, suppose that $\sum_{i'} A_{i'j_1} = 1$ and $\sum_{i'} A_{i'j_2} = 1$. Apply $j_2$-moves to $\{ f_{ij_1}(\alpha) \}$ until there are no $j_2$-moves available and apply $j_1$-moves to $\{ f_{ij_2}(\alpha) \}$ until there are no $j_1$-moves available. It is a straightforward proof by induction that these two processes have the same result. This completes the proof that $\calS_{\mathrm{final}}$ is unique, and thus $\Sigma$ is well-defined.

We next show that every element of $\Sigma(P_0,A_0)$ is simplex which is unimodular in $L_{P_0} \times \Lambda^{m-1}$. It is easy to see that if $(P,A,\phi) \in \calS_{\mathrm{final}}$, then $\Cay(AP)$ must be a unimodular simplex. Since each $\phi$ is a lattice polytope isomorphism to a polytope with ambient lattice $L_{P_0} \times \Lambda^{m-1}$, the result holds.

To show that $\Sigma$ is a triangulation, we note that $f_j$ is canonical in the following sense: Let $(F,A') \to (P,A)$ be a morphism in $\calF$. Assume $(P,A)$ is $j$-reducible. Then either $\Cay(A'F)$ is a face of one of the elements of $\Phi(f_j(P,A,\mathrm{id}))$, or $(F,A')$ is $j$-reducible and
\[
\Phi(f_j(F,A',\mathrm{id})) = \Phi(f_j(P,A,\mathrm{id})) |_{\Cay(A'F)}.
\]
where $|_{\Cay(A'F)}$ denotes the induced dissection on ${\Cay(A'F)}$. This implies that $\Sigma(F,A') = \Sigma(P,A)|_F$.

We can now prove $\Sigma(P,A)$ is a triangulation by induction on the volume of $\Cay(AP)$. If $\Cay(AP)$ is not already a unimodular simplex, then there is a sequence of moves which subdivides $\Cay(AP)$ into two full-dimensional polytopes; call their common facet $F$. By induction, $\Sigma$ gives triangulations of these two polytopes, and the aforementioned canonicity implies that these triangulations agree on $F$. Thus $\Sigma$ gives a triangulation of $\Cay(AP)$. This argument also shows that the triangulation is canonical. Projectivity is standard to check.
\end{proof}

\subsubsection{The second triangulation} \label{sec:polysubdv2}

This triangulation is similar to the first one but differs in the order we apply the operators $f_{ij}$. To describe the triangulation, we need some additional notation.

A \emph{stratified matrix} is a matrix $A$ of nonnegative integers with rows indexed by a set $I_A$ and columns indexed by $[n]$ for some $n$, along with an ordered partition $\Pi_A = (I_1,\dots,I_p)$ of $I_A$, where each $I_k$ is a an ordered set. For an ordinary matrix $A$ with $I_A = [m]$, we can view $A$ as a stratified matrix with the \emph{standard stratification} $\Pi_A = (\{1\},\{2\},\dots,\{m\})$. Analogously to the definition of $\calF$, we define $\tilde{\calF}$ to be the category whose objects are tuples $(P,A)$, except that $A$ is a stratified matrix. The morphisms are $(F,A') \to (P,A)$ where $F_i$ is a face of $P_i$ for all $i$, $A'$ is obtained from $A$ by restricting to a subset of rows of $A$, and if $\Pi_A = (I_1,\dots,I_p)$, then $\Pi_{A'} = (I_1 \cap I_{A'}, \dots, I_p \cap I_{A'})$. Note that the standard stratification makes $\calF$ a full subcategory of $\tilde{\calF}$. We also have a full and faithful functor $\tilde{\Cay} : \tilde{\calF} \to \calP$ given by $(P,A) \mapsto \Cay(AP)$ which restricts to $\Cay$ on $\calF$.

Let $(P,A) \in \tilde{F}$ with $\Pi_A = (I_1,\dots,I_p)$. Let $A_{I_k}$ be the ordinary matrix obtained by restricting $A$ to the rows indexed by $I_k$. We say that $A$ is \emph{$(I_k,j)$-reducible} if $A_{I_k}$ is $j$-reducible. If $A$ is $(I_k,j)$-reducible, then we define $f_{I_k,j}(P,A,\phi)$ to be the set obtained by applying $f_j$ locally to $(P,A_{I_k})$; in other words,
\[
f_{I_k,j}(P,A,\phi) := \{ (P',\tilde{A'},\phi \circ \tilde{\phi'}) : (P',A',\phi') \in f_j(P,A_{I_k},\mathrm{id}) \}
\]
where $\tilde{A'}$ is obtained from $A$ by replacing $A_{I_k}$ with $A'$ (and $I_k$ with $I_{A'}$ in $\Pi_A$), and $\tilde{\phi'} : \Cay(\tilde{A'}P') \to \Cay(AP')$ is the affine map restricting to $\phi'$ on $\Cay(A'P')$ and the identity on $\Cay(A_{I_l} P')$ for $l \neq k$. We say that $A$ is \emph{$I_k$-reducible} if it is $(I_k,j)$-reducible for some $j$.

Fix $(P_0,A_0) \in \tilde{F}$, and define a process as follows. We keep track of a set $\calS$, which is initially $\{(P_0,A_0,\mathrm{id})\}$. A \emph{move} consists of the following. First, choose an $\alpha = (P,A,\phi) \in \calS$. Let $\Pi_A = (I_1,\dots,I_p)$, and let $k$ be the smallest index for which $A$ is $I_k$-reducible. Choose a $j$ such that $A$ is $(I_k,j)$-reducible, and modify $\calS$ by replacing $\alpha$ with the elements of $f_{I_k,j}(\alpha)$. Apply such moves until there are no moves available.

Defining $\Phi(\calS)$ as before, we have the following.

\begin{lem} \label{lem:polysubdv2}
The above process always terminates and the final result $\calS_{\mathrm{final}}$ is independent of the moves chosen. Moreover, the map
\[
(P_0,A_0) \mapsto \Phi(\calS_{\mathrm{final}})
\]
is a projective canonical triangulation $\Sigma$ of $\tilde{\Cay}$ such that $\Sigma(P_0,A_0)$ is unimodular with respect to $L_{P_0} \times \Lambda^{\lvert I_A \rvert-1}$.
\end{lem}

The proof is identical to the proof of Lemma~\ref{lem:polysubdv1}. Also, note that the triangulation of Lemma~\ref{lem:polysubdv1} is obtained from this Lemma by viewing $A_0$ as a stratified matrix with $\Pi_{A_0} = ([m])$. So Lemma~\ref{lem:polysubdv2} can be seen as a generalization of Lemma~\ref{lem:polysubdv1} to $\tilde{\calF}$. However, we obtain a different triangulation than Lemma~\ref{lem:polysubdv1} if we use the standard stratification of $A_0$.

\subsection{Proof of Lemma~\ref{lem:prebox}} \label{sec:boxsubdv}

We are now ready to complete the proof of Theorem~\ref{thm:polytopes} by proving Lemma~\ref{lem:prebox}. Recall the definition of $\tilde{\calF}$ from Section~\ref{sec:polysubdv2}. We will focus on the following subcategory of $\tilde{\calF}$.

\begin{defn} \label{def:Caym}
Let $\bfm$ be a box point of $P_0$ for some $(P_0,A_0) \in \Ob(\tilde{\calF})$. We define $\tilde{\calF}^{\bfm}$ to be the full subcategory of $\tilde{\calF}$ whose objects are $(P,A) \in \Ob(\tilde{\calF})$ satisfying the following:
\begin{enumerate}
\item Let $c = c(P, \bfm)$. Then for all $i \in I_A$ and $j \in [n]$, we have
\[
A_{ij} = 0 \text{ or } A_{ij} \ge c_j.
\]
\item Let $\Pi_A = (I_1,\dots,I_p)$. Then for all $k \in [p]$ and all $i$, $i' \in I_k$, we have
\[
\supp A_{i} \cap \supp c = \supp A_{i'} \cap \supp c := J_k.
\]
Moreover, we have
\[
J_1 \supseteq J_2 \supseteq \dots \supseteq J_p.
\]
\end{enumerate}
Let $\tilde{\Cay}^{\bfm}$ be the restriction of $\tilde{\Cay}$ to $\tilde{\calF}^{\bfm}$.
\end{defn}

Clearly $\calF^{\bfm}$ is a subcategory of $\tilde{\calF}^{\bfm}$ via the standard stratification, and $\Cay^{\bfm}$ is the restriction of $\tilde{\Cay}^{\bfm}$ to this subcategory. It thus suffices to prove the following refinement of Lemma~\ref{lem:prebox}.

\begin{lem} \label{lem:box}
Let $\bfm$ be as above. Then there is a projective canonical triangulation $\Sigma^{\bfm}$ of $\tilde{\Cay}^{\bfm}$ such that for all $(P,A) \in \Ob(\tilde{\calF}^{\bfm})$ and all full-dimensional simplices $Q$ in $\Sigma^{\bfm}(P,A)$, we have the following.
\begin{itemize}
\item If $\bfm$ is a box point of $P_A$, then
\[
\ind(Q) < \ind(P_A).
\]
\item If $\bfm$ is not a box point of $P_A$, then $L_Q = L_{P_A} \times \Lambda^{\lvert I_A \rvert-1}$.
\end{itemize}
\end{lem}

\begin{proof}
Let $(P,A) \in \Ob(\calF^{\bfm})$. If $\bfm$ is not a box point of $P_A$, then we set $\Sigma^{\bfm}(P,A)$ to be the triangulation from Lemma~\ref{lem:polysubdv2}. This satisfies the desired properties.

Now assume $\bfm$ is a box point of $P_A$. We proceed by induction on the size of the set of configurations reachable from $\{(P,A,\mathrm{id})\}$ through moves, as defined in of Section~\ref{sec:polysubdv}. If no moves are available, then all entries of $A$ are 0. In this case $P_A = (\{0\})$ so $\bfm$ is not a box point of $P_A$, which we already considered. So we may assume there are moves available and $A$ is not a 0-matrix. Note that deleting a 0-column of $A$ and the corresponding entry of $P$ does not change $P_A$ or $\Cay(AP)$. So we may assume $A$ has no 0-columns and $P_A = P$.

Let $c := c(P,\bfm)$. Since $\bfm$ is a box point of $P$, we have $\supp c \neq \emptyset$. Let $\Pi_A = (I_1,\dots,I_p)$, and for each $k \in [p]$, define $J_k$ as in Definition~\ref{def:Caym}. Since $A$ has no 0-columns, and $J_1 \supseteq \dots \supseteq J_p$, we must have $J_1 = \supp c$. By the definition of $\tilde{\calF}$, this implies $A_{ij} \ge c_j$ for all $i \in I_1$ and $j \in [n]$.

We separate the argument into two cases.

\medskip
\emph{Case 1:} $A_{ij} > c_j$ for some $i \in I_1$ and $j \in [n]$.

Let $j \in [n]$ be the smallest index for which $A_{ij} > c_j$ for some $i \in I_1$. Then $A$ is $(I_1,j)$-reducible. Let $i \in I_1$ be the unique index such that $f_j(P,A_{I_1},\mathrm{id}) =  f_{ij}(P,A_{I_1},\mathrm{id})$. By our assumption and the definition of $i$, $A_{ij} > c_j$.

Now, recall that $\Phi(f_{I_1,j}(P,A,\mathrm{id}))$ subdivides $\Cay(AP)$ into one or two polytopes, and comes with an ordering on these polytopes. Call the first of these polytopes $Q_1$. Then we have an affine isomorphism $\phi_1 :\Cay(A'P) \to Q_1$, where either
\begin{enumerate} \renewcommand{\labelenumi}{(\alph{enumi})}
\item $A'$ is obtained from $A$ by subtracting 1 from the $A_{ij}$ entry, or
\item $A' = A^\square$ as defined in Section~\ref{sec:polysubdv1}.
\end{enumerate}
In either case, since $A_{ij} > c_j$, it is easy to check that $(P,A') \in \calF^{\bfm}$ and $\bfm$ is a box point of $P_{A'}$. Thus, by induction, we have a triangulation $\Sigma_1$ of $\Cay(A'P)$ into full-dimensional simplices each with index less than $\ind(P_{A'})$.

We now wish to show that $\ind(\phi_1(Q)) < \ind(P)$ for all full-dimensional $Q \in \Sigma_1$. If we have case (a), then $\phi_1$ is a lattice polytope isomorphism and $P_{A'} = P_A = P$, so the claim follows. Suppose we have case (b). Let $Q \in \Sigma_1$ be full dimensional. Then
\[
\ind(\phi_1(Q)) = \ind(P_j)\ind(Q).
\]
Also, $\ind(P) = \ind(P_j)\ind(P_{A'})$. Thus $\ind(\phi_1(Q)) < \ind(P)$, as desired.

If $f_j(P,A,\mathrm{id}))$  has one element, we are done. Otherwise, let $Q_2$ be the second element of $\Phi(f_j(P,A,\mathrm{id}))$. We need to give a triangulation of $Q_2$. Recall the definitions of $A''$, $v$, and $F$ from Section~\ref{sec:polysubdv1}. Then we have an affine isomorphism $\phi_2 : \Cay(A''F) \to Q_2$.

First suppose $\bfm$ is a box point of $F$. Since $c(F,\bfm) = c$ and $A_{ij} > c_j$, we have $(F,A'') \in \calF^{\bfm}$. We thus have a triangulation $\Sigma_2$ of $\Cay(A''F)$ using the inductive hypothesis. Let $Q \in \Sigma_2$ be full-dimensional; we wish to show that $\ind(\phi_2(Q)) < \ind(P)$. Since $\bfm$ is a box point of $F$, we have by the inductive hypthesis $\ind(Q) < \ind(F_{A''}) = \ind(F)$. Now, let $h$ be the lattice distance between the parallel hyperplanes $\AffSpan(v + A_i'F)$ and $\AffSpan(A_iF)$. Then
\[
\ind(\phi_2(Q)) = h\ind(Q).
\]
Also, $h$ is the lattice distance between $v$ and $F_j$, so $\ind(P) = h \ind(F)$. We conclude that $\ind(\phi_2(Q)) < \ind(P)$, as desired.

Now assume $\bfm$ is not a box point of $F$. In particular, this implies $j \in \supp c$. Let $m$ be the unique representative of $\bfm$ in $cP$. Consider the polytope
\[
m + (A_i-c) F
\]
which is well-defined since $A_{ik}-c_{k} \ge 0$ for all $k$. The affine span of this polytope is parallel to $\AffSpan(v+A_i'F)$ and $\AffSpan(A_iF)$, and since $\bfm$ is not a box point of $F$, it lies strictly between these two hyperplanes. Moreover, this polytope is contained in $\conv((v+A_i'F)\cup(A_iF))$.

We use this to construct a subdivision of $Q_2$ as follows. First, we have the polytopes
\begin{align*}
R_1 &:= \Cay( (AF)[i, \conv((m+(A_i-c)F)\cup(A_iF))] ) \\
R_2 &:= \Cay( (AF)[i, \conv((m+(A_i-c)F)\cup(v+A_i'F))] ).
\end{align*}
In addition, let $\mathfrak{F}$ denote the set of all $F' = (F_1' \dots F_n')^T$ such that $F_k'$ is a face of $P_k$ for all $k$, $\sum_k F_k'$ is a facet of $\sum_k P_k$, and $\bf m$ is not a box point of $F'$. Let $\mathfrak G$ denote the set of all $G = (G_1 \dots G_n)^T$ such that $G_k$ is a face of $F_k$ for all $k$, $\sum_k G_k$ is a facet of $\sum_k F_k$, and such that there exists $F' \in \mathfrak{F}$ with
\[
G_k = F_k \cap F_k'
\]
for all $k$. For each $G \in \mathfrak{G}$, we define the polytope
\[
R_G := \Cay( (AG)[i, \conv((m+(A_i-c)G)\cup(v+A_i'G)\cup(A_iG))]).
\]
Then $\{R_1,R_2\}\cup\{R_G\}_{G\in\mathfrak{G}}$ forms a subdivision of $Q_2$.

Let $B$, $B'$, and $B''$ be the matrices obtained by inserting $A_i-c$ above the $i$-th row of $A$, $A'$, and $A''$, respectively (this operation is done locally in $I_1$). Then we have affine isomorphisms
\[
R_1 \cong \Cay(F,B) \qquad
R_2 \cong \Cay(F,B') \qquad
R_G \cong \Cay(G,B'')
\]

Using Lemma~\ref{lem:polysubdv2}, we obtain triangulations $\Xi_1$, $\Xi_2$, and $\Xi_G$ of $R_1$, $R_2$, and $R_G$. We prove that these triangulations give the desired triangulation of $Q_2$. Suppose $Q$ is a full-dimensional simplex in $\Xi_1$. Then
\[
\ind(Q) = h' \ind(F)
\]
where $h'$ is the lattice distance between $\AffSpan (m+(A_i-c)F)$ and $\AffSpan (A_iF)$. Since $h' < h$, where $h$ is as defined previously, we have $\ind(Q) < \ind(P)$, as desired. Similarly, $\ind(Q) < \ind(P)$ for all full-dimensional $Q \in \Xi_2$.

Finally, suppose $Q \in \Xi_G$ is full-dimensional. Then
\[
\ind(Q) = h'' \ind(F')
\]
where $F' \in \mathfrak{F}$ is such that $G_k = F_k \cap F_k'$ for all $k$, and $h''$ is the lattice distance between $m$ and $\sum_k c_k F'_k$. By replacing $F$ with $F'$ in the previous arguments, we have that $\ind(Q) < \ind(P)$, as desired.

All that remains is to show canonicity. Let $\Sigma^{\bfm}$ denote the above triangulation. Let $(F,A') \to (P,A)$ be a morphism in $\tilde{\calF}$. We wish to show that $\Sigma^{\bfm}(F,A') = \Sigma(P,A)|_{\Cay(A'F)}$. If $\Cay(A'F)$ is a face of one of $Q_1$ and $Q_2$, we are done by induction. Assume otherwise. First suppose $\bfm$ is a box point of $F$. Then it is straightforward to check that if we repeat the above process with $(P,A)$ replaced by $(F,A')$, we get $\Sigma^{\bfm}(P,A)|_{\Cay(A'F)}$. The crucial detail is that since $c(F,\bfm) = c(P,\bfm)$, we end up choosing the same values for $i$ and $j$.

Now assume $\bfm$ is not a box point of $(F,A')$. One can check that $\Sigma^{\bfm}(P,A)|_{\Cay(A'F)}$ is the triangulation obtained by triangulating each element of $\Phi(f_{I_1,j}(F,A',\mathrm{id}))$ by the inductive hypothesis, i.e. by the triangulation $\Sigma$ from Lemma~\ref{lem:polysubdv2}. Then Lemma~\ref{lem:polysubdv2} implies that this triangulation is $\Sigma(F,A')$. Hence $\Sigma(P,A)|_{\Cay(A'F)} = \Sigma(F,A') = \Sigma^{\bfm}(F,A')$, as desired.

\medskip
\emph{Case 2:} $A_{i} = c$ for all $i \in I_1$.

Define $m$ and $\mathfrak{F}$ as in the previous case. Let $i$ be the first element of $I_1$. Then $\Cay(AP)$ has a subdivision into the full-dimensional polytopes
\[
R_F := \Cay( (AF)[i, \conv( \{m\} \cup cF )] )
\]
where $F$ ranges over all elements of $\mathfrak{F}$, and
\[
R_A := \Cay( (AP)[i, \{m\}])
\]
if $R_A$ is full-dimensional. Let $A'$ be the stratified matrix obtained by deleting the $i$-th row of $A$. Then we have affine isomorphisms $\phi_F : \Cone( \Cay(AF) ) \to R_F$ and $\phi_A : \Cone( \Cay(A'P) ) \to R_A$, where
\[
\Cone(Q) := \Cay( \{0\}, Q ).
\]

By Lemma~\ref{lem:polysubdv2}, we obtain a triangulation of $\Cay(AF)$, which gives a triangulation of $R_F$ by coning and $\phi_F$. Similar arguments to the previous section imply all full-dimensional simplices of this triangulation have index less than $\ind(P)$, as desired. Finally, if $R_A$ is full-dimensional, then $P_{A'} = P_A = P$ and $(P,A') \in \tilde{\calF}$, so by induction we have a triangulation of $\Cay(A'P)$ whose full-dimensional simplices have index less than $\ind(P)$. Coning this triangulation gives the desired triangulation of $R_A$. Canonicity is easy to check; we note that in this case, if $(F,A') \to (P,A)$ is a morphism and $\bfm$ is not a box point of $F$, then $\Cay(A'F)$ is a face of some $R_F$.
\end{proof}

\section{More semistable reduction theorems}\label{ssrsec}

\subsection{The main theorem for log schemes}
Our goal now is to lift the canonical alteration of fans constructed in Theorem~\ref{mainth} to log schemes.

\subsubsection{Toric stacks}
It is well known that subdivisions of fans naturally lift to modifications of log schemes (in the old language this was worked out in \cite{KKMS}). Lifting more general morphisms, including the case of alterations that we need, involves toric stacks that we briefly recall now.

For a scheme $X$ and a lattice $\Lambda$ let $T_{X,\Lambda}=X\times\Spec(\bfZ[\Lambda])$ denote the corresponding $X$-torus. Given a fine monoid $P$ let $F_P=\Spec(P)$, $Z_P=\Spec(\bfZ[P])$ and $\calZ_P=[Z_P/T_{P^\gp}]$ denote the corresponding affine fan, toric scheme and toric stack, respectively. To a log scheme $X$ with a global chart $X\to Z_P$ and a homomorphism of fine monoids $\phi\:P\to Q$ we associate the relative toric scheme $X_P[Q]=X\times_{Z_P}Z_Q$ and the relative toric stack $$\calX_P[Q]=X\times_{\calZ_P}\calZ_Q=[X_P[Q]/T_{X,Q^\gp/P^\gp}].$$

\begin{rem}\label{toricstackrem}
(i) The stack $\calX_P[Q]$ plays an important role in logarithmic geometry, and it only depends on the fan $\oP\to\oM_X$ and the sharpening $\ophi\:\oP\to\oQ$, see \cite[Proposition~5.17]{Olsson-logarithmic} or \cite[Lemma~3.2.4]{Molcho-Temkin}. Thus, $\calX_P[Q]$ can be viewed as the base change of the map of fans $F_Q\to F_P$ and the dependence on $Q$ is only through $\oQ$.

(ii) Unlike $\calX_P[Q]$, the scheme $X_P[Q]$ essentially depends on the chart $P\to M_X$ and the homomorphism $\phi$. So it is more correct to view it as a base change of a morphism of monoschemes. This technical point was essential in \cite[Section~3]{Illusie-Temkin} and \cite{ALPT} but not in this paper since we will only work with $\calX_P[Q]$.

(iii) If $P^\gp=Q^\gp$, then $X_P[Q]=\calX_P[Q]$.
\end{rem}

\subsubsection{Globalization}
Assume that $F_R\into F_Q$ is an open immersion, that is, $\oQ\to\oR$ is a sharpened localization. Then it is easy to see that $\calX_P[R]\to\calX_Q[R]$ is an open immersion too. It follows that the construction globalizes to fans: to any fan $F$ over $F_P$ one associates a toric stack $\calX_P[F]$ with the natural log structure.

\subsubsection{Lifting fan functors}
Let $\calL$ be a construction that associates to any fan $F$ a morphism of fans $\calL(F)\to F$. We say that $\calL$ is {\em quasi-local} if for any surjective local isomorphism $f\:F'\to F$ we have that $\calL(F')=\calL(F)\times_FF'$. The following theorem is proved in \cite{Molcho-Temkin}. We briefly recall the argument, because a similar reasoning will be used later.

\begin{theor}\label{liftth}
Let $\calL$ be a quasi-local construction on fans. Then there exists a unique construction $\calL^\rmlog$ associating to any log scheme $X$ a morphism $\calL^\rmlog(X)\to X$ and such that the following two conditions are satisfied:

(i) If $Y\to X$ is a strict surjective morphism, then $\calL^\rmlog(Y)=\calL^\rmlog(X)\times_YX$.

(ii) If $X=\coprod_i X_i$ and $X_i\to F_{P_i}$ are global fans, then $\calL^\rmlog(X)=\coprod(\calX_i)_{P_i}[F_i]$, where $F_i$ is the preimage of $F_{P_i}$ in $\calL(\coprod_i F_{P_i})$.
\end{theor}
\begin{proof}
If $X=\coprod_i X_i$ and each $X_i$ possesses a global affine fan, then (ii) dictates the definition of $\calL^\rmlog(X)$, and the independence of choices follows from Remark~\ref{toricstackrem}(i). In the general case, find a strict \'etale covering $\coprod_iX_i\to X$, such that each $X_i$ possesses a global chart, and use (i) and \'etale descent.
\end{proof}

\begin{exam}
(i) If $\calL(X)\to X$ is a subdivision, then it follows from Remark~\ref{toricstackrem}(iii) that $\calL^\rmlog(X)$ is a scheme, which is \'etale locally of the form $X_P[Q]$, and $\calL^\rmlog(X)\to X$ is a log \'etale modification. For example, in this way one can define the barycentric subdivision $X_{\rm bar}$ of the monoidal structure of $X$. Similarly, there are various resolution of fans $F\mapsto F_{\rm res}$ constructed in the literature, and some of them are local constructions. Lifting such a construction one obtains a monoidal resolution $X\mapsto X_{\rm res}$ and a log \'etale modification $X_{\rm res}\to X$. The monoids $\oM_{X_{\rm res},x}$ are free, in particular, if $X$ is log regular then $X_{\rm res}$ is log regular and the underlying scheme is regular. Since \cite{KKMS} such constructions are used to resolve toroidal singularities, see \cite[Section~3.3]{Illusie-Temkin} for a modern treatment.

(ii) Fix $c>0$. For an fs monoid $P$ the group $\oP^\gp$ is torsion free and we let $c^{-1}\oP$ denote the saturation of $\oP$ in $c^{-1}\oP^\gp$. Setting $cF_P=F_{c^{-1}\oP}$ we obtain a functor compatible with localizations and hence extending to a local functor $F\mapsto cF$ on the category of saturated fans. The induced functor on fs log schemes is nothing else but the root stack introduced in \cite{rootstacks} and denoted there $X_c$. Note that $X_c\to X$ is a proper log \'etale morphism, which is non-representable whenever $c>1$ and the log structure is non-trivial.

(iii) In the same manner, if $\calL(X)\to X$ is an alteration, then $\calL^\rmlog(X)\to X$ is a proper log \'etale morphism, which is an isomorphism over the triviality locus of the log structure. We call such a morphism {\em monoidal alteration}.
\end{exam}

\subsubsection{Monoidal semistable reduction}
A morphism of log schemes $f\:Y\to X$ is called {\em monoidally semistable} if $Y$ and $X$ are monoidally regular and $f\:Y\to X$ is saturated. By a {\em (projective) monoidal resolution} of a morphism $Y\to X$ of log schemes we mean a (projective) monoidal alteration $X'\to X$ and a (projective) monoidal subdivision $Y'\to Y\times_XX'$ such that $f'\:Y'\to X'$ is monoidally semistable.

\begin{theor}\label{mainlogth}
There exists a construction associating to each morphism of fine log schemes $f\:X\to B$ a projective monoidal resolution $f'\:X'\to B'$ and compatible with surjective strict morphisms: if $g\:Y\to C$ is another morphism and $\alpha\:Y\to X$, $\beta\:C\to B$ are strict surjective and such that $\beta\circ g=f\circ\alpha$, then the monoidal resolution $g'$ of $g$ is the base change of $f'$ in the sense that $C'=C\times_BB'$ and $Y'=Y\times_XX'$.
\end{theor}
\begin{proof}
This is done similarly to the proof of Theorem~\ref{liftth}. The only difference is that this time we lift the quasi-local construction that associates a morphism of fans to a morphism of fans.

Assume first that $f$ possess a quasi-local chart in the following sense: $B=\coprod_i B_i$ with charts $B_i\to Z_{P_i}$, $X=\coprod_{ij} X_{ij}$ with charts $Z_{Q_{ij}}$, and $f$ restricts to morphisms $X_{ij}\to B_i$ with charts $Z_{Q_{ij}}\to Z_{P_i}$. Consider the corresponding map of fans $g\:\coprod_{ij}F_{Q_{ij}}\to\coprod_i F_{P_i}$ and let $g'$ be its resolution constructed in Theorem~\ref{mainth}. The source (resp. target) of $g'$ is an alteration of the source (resp. target) of $g$, hence they have the same decomposition to connected components, say  $g'\:\coprod_{ij}F_{ij}\to\coprod F_i$. Clearly, the morphism $f'\:X'\to B'$, where  $B'=\coprod_i (\calB_i)_{P_i}[F_i]$ and $X'=\coprod_{ij} (\calX_{ij})_{Q_{ij}}[F_{ij}]$, is a resolution of $f$, and its independence of the chart follows from Remark~\ref{toricstackrem}(i) and quasi-locality of the resolution $g'$ of $g$. Moreover, by the same reason this construction is compatible with surjective strict morphisms.

In general, there exist surjective strict \'etale morphisms $X_0\to X$ and $B_0\to B$ with a morphism $f_0\:X_0\to B_0$ compatible with $f$ such that $f_0$ possesses a quasi-local chart. The induced morphism $f_1\:X_0\times_XX_0\to B_0\times_BB_0$ also possesses a quasi-local chart (for example, either pullback of the chart of $f_0$), hence the resolution of $f_0$ descends to a resolution of $f$ by descent.
\end{proof}

\subsection{Applications}\label{finalsec}

\subsubsection{Semistable morphisms}\label{semistabsec}
We say that a morphism of log schemes $f\:X\to B$ is {\em semistable} if $f$ is monoidally semistable, and in addition $B$ is log regular and $f$ is log smooth. This happens if and only if the following conditions hold:

\begin{itemize}
\item[(i)] $X$ and $B$ are regular and the log structures are given by normal crossings divisors $Z\into X$ and $W\into B$.

\item[(ii)] \'Etale-locally at any $x\in X$ with $b=f(x)$ there exist regular parameters $t_1\.. t_n,t'_1\.. t'_{n'}\in\calO_x$ and $\pi_1\..\pi_l,\pi'_1\..\pi'_{l'}\in\calO_b$ such that $Z=V(t_1\ldots t_n)$ at $x$, $W=V(\pi_1\ldots\pi_l)$ at $b$, $f^\#(\pi_i)=t_{n_i+1}\ldots t_{n_{i+1}}$ for $0=n_1<n_2<\ldots<n_{l+1}\le n$,

\item[(iii)] $f$ is log smooth. 
\end{itemize}
In characteristic zero, (iii) can be replaced by the condition that $f^\#(\pi'_j)=t'_j$ for $1\le j\le l'$. (In positive characteristic, this is neither necessary, nor sufficient.)

\begin{theor}\label{semistableth}
Assume that $f\:X\to B$ is a log smooth morphism between fine log schemes. Then there exists a monoidal alteration $b\:B'\to B$ and a monoidal subdivision $a\:X'\to X\times_BB'$ such that the morphism $f'\:X'\to B'$ is semistable. 
\end{theor}
\begin{proof}
We simply apply Theorem~\ref{mainlogth} to $f$. Since $a$ and $b$ are log \'etale, this implies that $f'$ is log smooth and monoidally semistable, that is, $f'$ is semistable.
\end{proof}

\begin{rem}
If $X$ and $B$ are schemes, then using Kawamata's trick as in \cite{AK}, one can achieve that $X'$ and $B'$ are also schemes. However, in this case $b$ can only be taken to be a usual alteration, and its choice is non-canonical. In particular, it does not have to be nice (e.g. \'etale) over the triviality locus of the log structure of $B$.
\end{rem}

The above theorem and the toroidalization theorem of Abramovich-Karu imply the semistable conjecture. We will use a finer toroidalization proved in \cite{ATW-relative} to obtain a strong version of the semistable reduction conjecture, where one also controls the modification loci, see part (ii) below. For convenience of applications we use closed subsets in the formulation, but one can easily reformulate this result in the language of log schemes.

\begin{theor}\label{mainapp}
Assume that $X\to B$ is a dominant morphism of finite type between qe integral schemes of characteristic zero and $Z\subsetneq X$ is a closed subset. Then there exists a stack-theoretic modification $b\:B'\to B$, a projective modification $a\:X'\to (X\times_BB')^{\rm pr}$, and divisors $W'\into B'$, $Z'\into X'$ such that:

(i) $a^{-1}(Z)\cup f'^{-1}(W')\subseteq Z'$ and the morphism $f'\:(X',Z')\to(B',W')$ is semistable. In particular, $X',B'$ are regular and $Z',W'$ are snc.

(ii) If a regular open $B_0\subseteq B$ is such that $X_0=X\times_BB_0\to B_0$ is smooth and $Z_0=Z\times_BB_0$ is a relative divisor over $B_0$ with normal crossings (in other words, $(X_0,Z_0)\to B_0$ is semistable), then $a$ and $b$ are isomorphisms over $X_0$ and $B_0$, respectively.
\end{theor}
\begin{proof}
Replacing $X$ by the blow up along $Z$ one does not modify $X_0$. So, one can assume that $Z$ is a divisor, and then $X$ possesses a log structure $M_X\into\calO_X$ whose triviality locus is $X\setminus Z$. Choose such an $M_X$ and provide $B$ with the trivial log structure. Resolving the morphism of log schemes $f\:X\to B$ by the main result of \cite{ATW-relative}, we find modifications of log schemes $b''\:B''\to B$ and $a''\:X''\to X$ with a log smooth morphism $f''\:X''\to B''$ compatible with $f$. By the loc.cit., if $f$ is log smooth over an open $B_1\subseteq B$, then one can achieve that $b''$ and $a''$ are isomorphisms over $B_1$ and $X\times_BB_1$, respectively. This applies to $B_0=B_1$ because $(X_0,Z_0)\to B_0$ is even semistable. In addition, resolving $B''$ we can also achieve that it is log regular. Again, this does not modify $B_0$ because it is regular. Now, it remains to apply Theorem~\ref{semistableth} to $f'$.
\end{proof}

In the same way, using \cite[Theorem~4.3.1]{tame-distillation} one obtains altered semistable reduction in all characteristics.

\begin{theor}\label{mainapp1}
Let $X\to B$ be a dominant morphism of finite type between integral schemes and $Z\subsetneq X$ a closed subset. Assume that $B$ is of finite type over a qe scheme of dimension at most 3. Then there exists a stack-theoretic ${\rm char}(B)$-alteration $b\:B'\to B$, a projective ${\rm char}(B)$-alteration $a\:X'\to (X\times_BB')^{\rm pr}$, and divisors $W'\into B'$ and $Z'\into X'$, such that $a^{-1}(Z)\cup f'^{-1}(W') \subseteq Z'$ and the morphism $f'\:(X',Z')\to(B',W')$ is semistable.
\end{theor}

In both results one can use the Kawamata's trick to achieve that $b$ is also projective at cost of making it an alteration of an arbitrary degree that modifies $B_0$.

\bibliographystyle{myamsalpha}
\bibliography{log_smooth}

\providecommand{\bysame}{\leavevmode\hbox to3em{\hrulefill}\thinspace}
\providecommand{\MR}{\relax\ifhmode\unskip\space\fi MR }
\providecommand{\MRhref}[2]{%
  \href{http://www.ams.org/mathscinet-getitem?mr=#1}{#2}
}
\providecommand{\href}[2]{#2}
\begin{thebibliography}{KKMS73}

\bibitem[AK00]{AK}
Dan Abramovich and Kalle Karu, \emph{Weak semistable reduction in
  characteristic 0}, Invent. Math. \textbf{139} (2000), no.~2, 241--273.
  \MR{1738451 (2001f:14021)}

\bibitem[ATW19]{ATW-relative}
Dan Abramovich, Michael Temkin, and Jaros{\l}aw W{\l}odarczyk, \emph{Relative
  desingularization and principalization of ideals}, in preparation.

\bibitem[ALPT18]{ALPT}
Karim Adiprasito, Gaku Liu, Igor Pak, and Michael Temkin, \emph{Log smoothness
  and polystability over valuation rings}, arxiv:1806.09168.

\bibitem[Ale94]{alex1}
Valery Alexeev, \emph{Boundedness and {$K^2$} for log surfaces}, Internat. J.
  Math. \textbf{5} (1994), no.~6, 779--810.

\bibitem[Ale96]{alex2}
\bysame, \emph{Moduli spaces {$M_{g,n}(W)$} for surfaces}, Higher-dimensional
  complex varieties ({T}rento, 1994), de Gruyter, Berlin, 1996, pp.~1--22.

\bibitem[BV12]{rootstacks}
Niels Borne and Angelo Vistoli, \emph{Parabolic sheaves on logarithmic
  schemes}, Adv. Math. \textbf{231} (2012), no.~3-4, 1327--1363. \MR{2964607}

\bibitem[dJ97]{dejong-curves}
Aise~Johan de~Jong, \emph{Families of curves and alterations}, Ann. Inst.
  Fourier (Grenoble) \textbf{47} (1997), no.~2, 599--621. \MR{1450427
  (98f:14019)}

\bibitem[DM69]{DM}
P.~Deligne and D.~Mumford, \emph{The irreducibility of the space of curves of
  given genus}, Inst. Hautes \'{E}tudes Sci. Publ. Math. (1969), no.~36,
  75--109. \MR{0262240}

\bibitem[HPPS14]{Santos}
Christian Haase, Andreas Paffenholz, Lindsay~C. Piechnik, and Francisco Santos,
  \emph{{Existence of unimodular triangulations - positive results.}}, to
  appear in Mem. Am. Math. Soc., arXiv:1405.1687.

\bibitem[IT14a]{Illusie-Temkin}
Luc Illusie and Michael Temkin, \emph{Expos\'e {VIII}. {G}abber's modification
  theorem (absolute case)}, Ast\'erisque (2014), no.~363-364, 103--160, Travaux
  de Gabber sur l'uniformisation locale et la cohomologie {\'e}tale des
  sch{\'e}mas quasi-excellents. \MR{3329777}

\bibitem[IT14b]{X}
\bysame, \emph{Expos\'e {X}: Gabber's modification theorem (log smooth case)},
  Ast\'erisque (2014), no.~363-364, 169--216, Travaux de Gabber sur
  l'uniformisation locale et la cohomologie {\'e}tale des sch{\'e}mas
  quasi-excellents. \MR{3329777}

\bibitem[Kar00]{Karu}
K.~Karu, \emph{Semistable reduction in characteristic zero for families of
  surfaces and threefolds}, Discrete Comput. Geom. \textbf{23} (2000), no.~1,
  111--120.

\bibitem[KKMS73]{KKMS}
George Kempf, Finn~F. Knudsen, David Mumford, and Bernard {Saint-Donat},
  \emph{Toroidal embeddings. {I}}, Springer-Verlag, Berlin, 1973, Lecture Notes
  in Mathematics, Vol.~339.

\bibitem[Mol16]{Molcho}
Sam Molcho, \emph{Universal weak semistable reduction}, arxiv:1601.00302.

\bibitem[MT19]{Molcho-Temkin}
Sam Molcho and Michael Temkin, \emph{Lagarithmically regular morphisms}, in
  preparation.

\bibitem[Ols03]{Olsson-logarithmic}
Martin~C. Olsson, \emph{Logarithmic geometry and algebraic stacks}, Ann. Sci.
  \'Ecole Norm. Sup. (4) \textbf{36} (2003), no.~5, 747--791. \MR{2032986}

\bibitem[Tem17]{tame-distillation}
Michael Temkin, \emph{Tame distillation and desingularization by
  {$p$}-alterations}, Ann. of Math. (2) \textbf{186} (2017), no.~1, 97--126.
  \MR{3665001}

\end{thebibliography}

\end{document}